\documentclass[reqno]{amsart}
\usepackage{amssymb}
\usepackage{hyperref}

\thanks{\copyright 2008 Universiteti i Prishtin\"es, Prishtin\"e, Kosov\"e.}
\vspace{9mm}

\newcommand{\B}{\mathfrak B}
\newcommand{\A}{\mathcal A}

\newcommand{\bea}{\begin{eqnarray*}}
\newcommand{\eea}{\end{eqnarray*}}
\numberwithin{equation}{section}
\begin{document}
\title[ WEAKLY ALMOST PERIODIC...]
{WEAKLY ALMOST PERIODIC BANACH ALGEBRAS ON SEMI-GROUPS  }

\author[B.Khodsiani, A.Rejali]
{Bahram Khodsiani, Ali Rejali}  

\address{Bahram Khodsiani\newline
 Department of Mathematics, Isfahan University ,  IRAN}
\email{ ${\texttt {b}}_{-}$khodsiani@sci.ui.ac.ir}

\address{Ali Rejali\newline
Department of Mathematics, Isfahan University ,  IRAN.\newline
}
\email{rejali@sci.ui.ac.ir }

\thanks{Submitted xxxxx   xxxxx. Published xxxxx xxxx.}

\thanks{Supported by the Center of Excellence for Mathematics at the Isfahan university}


\begin{abstract}
Let $\mathcal{WAP}(\A)$ be the space of all weakly almost periodic functionals on a  Banach algebra $\A$. The Banach algebra $\A$ for which the natural embedding of $\A$ into $\mathcal{WAP}(\A)^*$ is bounded below is called a $\mathcal{WAP}$-algebra. We show that  the second dual of a Banach algebra $\A$ is a $\mathcal{WAP}$-algebra, under each Arens products, if and only if $\A^{**}$ is a dual Banach algebra. This is  equivalent to the Arens regularity of $\A$. For a locally compact foundation semigroup $S$, we show that the absolutely continuous  semigroup measure algebra  $M_a(S)$   is a $\mathcal{WAP}$-algebra if and only if the measure algebra $M_b(S)$ is so.
\end{abstract}

\maketitle
\numberwithin{equation}{section}
\newtheorem{theorem}{Theorem}[section]
\newtheorem{lemma}[theorem]{Lemma}
\newtheorem{proposition}[theorem]{Proposition}
\newtheorem{corollary}[theorem]{Corollary}
\newtheorem*{remark}{Remark}
\newtheorem{examples}[theorem]{Examples}

\section{Introduction and Preliminaries}

The dual $\A^*$ of a Banach algebra $\A$ can be turned into a Banach $\A-$module in a natural way, by setting

$$\langle  f\cdot a, b\rangle=\langle f ,ab\rangle \ \ \ \ \  {\rm and}    \ \ \ \ \langle a\cdot f, b\rangle=\langle f, ba\rangle\ \ \ \ (a,b\in \A, f\in \A^*).$$

A {\it dual Banach algebra} is a Banach algebra $\A$ such that $\A=(\A_*)^*$, as a Banach space,
for some Banach space $\A_*$, and such that $\A_*$ is a closed $\A-$submodule of $\A^*;$ or equivalently,
the multiplication on $\A$ is separately weak*-continuous. It should be remarked that the predual of a dual Banach algebra need not be unique, in general (see \cite{D, DPW}); so we usually point to the involved predual of a dual Banach algebra.

A functional $f\in \A^*$ is said to be {\it weakly almost periodic} if $\{f\cdot a: \|a\|\leq 1\}$
 is relatively weakly  compact in $\A^*$. We denote by $\mathcal{WAP}(\A)$ the set of all weakly
almost periodic elements   of $\A^*.$ It is easy to verify that, $\mathcal{WAP}(\A)$ is a (norm) closed subspace of $\A^*$. For more about weakly almost periodic functionals, see \cite{DL}.

It is known that the multiplication of a Banach algebra $\A$ has    two natural but, in general, different extensions (called Arens products) to the second dual $\A^{**}$ each turning $\A^{**}$  into a Banach algebra. When these extensions are equal, $\A$ is said to be (Arens) regular. It can be verified that $\A$ is Arens regular if and only if $\mathcal{WAP}(\A)=\A^*$. Further  information  for the Arens regularity of Banach algebras can be found in \cite{D, DL}.

$\mathcal{WAP}$-algebras, as a generalization of the Arens regular algebras, has been introduced and intensively studied in \cite{Da2}. Indeed, a Banach algebra $\A$ for which the natural embedding of $\A$ into $\mathcal{WAP}(\A)^*$ is bounded below, is called a $\mathcal{WAP}$-algebra. It has also known that $\A$ is a $\mathcal{WAP}$-algebra if and only if it admits an isomorphic representation on a reflexive Banach space. If $\A$ is a $\mathcal{WAP}$-algebra, then $\mathcal{WAP}(\A)$ separate the points of $\A$ and  so it is $\omega^*$-dense in $\A^*$.

It can be readily verified that every dual Banach algebra, and every Arens regular Banach algebra, is a $\mathcal{WAP}$-algebra for comparison see \cite{Bkh}. Moreover group algebras are also always $\mathcal{WAP}$-algebras, however,  they are neither dual Banach algebras, nor Arens regular in the case where the underlying group is not discrete, see \cite[Corollary3.7]{B} and \cite{y1,Bkh}.  Ample information about $\mathcal{WAP}$-algebras with further details can be found in the impressive paper \cite{Da2}.

The paper is organized as follows.  In section 2 we introduce some algebraic and topologic properties of $\mathcal{WAP}$-algebras.  How to inherit the weakly almost periodicity of Banach algebras to the ideals, quotient spaces, direct sums  and by homomorphisms are studied. We show that  the second dual of a Banach algebra $\A$ is a $\mathcal{WAP}$-algebra, under each Arens products, if and only if $\A$ is  Arens regular.
 In section 3 we deal with a locally compact foundation semigroup $S$, we show that the absolutely continuous  semigroup measure algebra  $M_a(S)$   is a $\mathcal{WAP}$-algebra if and only if the measure algebra $M_b(S)$ is so.
\section{{ Some algebraic and topologic properties of $\mathcal{WAP}$-algebras}}

 N. J. Young turns out a Banach algebra ${\A}$ admits an isometric representation on a reflexive Banach
  space if and only if the weakly almost periodic functionals of norm one  determines the norm on ${\A}$.
The following proposition provides a property qualification of $\mathcal{WAP}$-algebras.
\begin{proposition}\label{equivalent} Let ${\A}$ be a Banach
 algebra. 
 Then the following statements  are equivalent.
\begin{enumerate}
  \item[(i)]
 ${\A}$ is a $\mathcal{WAP}$- algebra.
  \item[(ii)]
There is a reflexive Banach space $E$ and a bounded below continuous representation $\pi:  {\A}\rightarrow B(E)$.
\item[(iii)]There is a dual Banach algebra (or $\mathcal{WAP}$-algebra) ${\B}$ and a bounded below continuous homomorphism $\phi:  {\A}\rightarrow {\B}$.

\end{enumerate}
\end{proposition}
\begin{proof} (i)$\Rightarrow$(ii): The canonical map
$\iota:{\A} \rightarrow \mathcal{WAP}({\A})^*$ is bounded
below and  $\mathcal{WAP}({\A})^*$ is dual Banach algebra. According to
\cite[Corollary3.8]{Da2} there is a reflexive Banach space $E$ such that $\pi:\mathcal{WAP}({\A})^*\rightarrow B(E)$ is an isometric representation.
So $\pi\circ\iota:{\A}\rightarrow B(E)$ is a bounded below
continuous representation.

(ii)$\Rightarrow$ (iii) Trivial.

(iii)$\Rightarrow$ (i): Let  $\iota:{\A} \rightarrow \mathcal{WAP}({\A})^*$ be the canonical map. Suppose that  ${\B}$ is   a dual Banach algebra and $\phi:  {\A}\rightarrow {\B}$ is a  bounded below homomorphism, then
$\phi^*(\mathcal{WAP}({\B}))\subseteq \mathcal{WAP}({\A})$. Also:
\begin{eqnarray*}
  \|\iota(a)\| &=&\sup \{|f(a)|: f\in \mathcal{WAP}({\A}) \quad  and \quad \|f\|\leq 1\} \\
   &\geq&\sup \{|\phi^*(g)(a)|: g\in \mathcal{WAP}({\B}) \quad  and \quad \|\phi\|\|g\|\leq 1\} \\
   &=& \frac{1}{\|\phi\|}\|\phi(a)\|\quad \mbox{(since }{\B} \rightarrow \mathcal{WAP}({\B})^*\ \mbox{is isometric)}
\end{eqnarray*}



So $\iota$ is bounded below.\end{proof}
\begin{corollary}\label{equiva}Let 
 $\|.\|_1$ and $\|.\|_2$ be equivalent norm on  algebra ${\A}$. Then $({\A},\|.\|_1)$ is a  $\mathcal{WAP}$-algebra if and only if $({\A},\|.\|_2)$ is so.
\end{corollary}
\begin{proof} Since the identity map is bounded below,  it is  an immediate consequence of proposition \ref{equivalent} (iii).
\end{proof}
By above corollary, it is convenient to suppose a $\mathcal{WAP}$-algebra $\A$  has an isometric representation on reflexive space $E$.  This can clearly be achieved by  renorming  $\A$ by $\|a\|_1=\|\pi(a)\|, (a\in \A)$ where $\pi$ is a bounded below representation on $E$.
\begin{corollary}\label{WAP-algebra}Let ${\A}$ be a $\mathcal{WAP}$-algebra and ${\B}$ be a closed subalgebra of ${\A}$. Then ${\B}$ is a $\mathcal{WAP}$-algebra.
\end{corollary}
\begin{proof} Since  the inclusion map is bounded below,  it is  an immediate consequence of proposition \ref{equivalent} (iii).
\end{proof}
\begin{proposition}\label{EEE}
Let $E$ be a Banach space. Then the following  statements are
equivalent:
\begin{enumerate}
  \item [(i)]$B(E)$ is  a dual Banach algebra.
  \item [(ii)]$B(E)$ is  a $\mathcal{WAP}$- algebra.
  \item [(iii)]$E\hat\otimes E^*\subseteq \mathcal{WAP}(B(E))$.
  \item[(iv)] $E$ is a reflexive Banach space.
\end{enumerate}
  \end{proposition}
  \begin{proof}
(i)$\Rightarrow $ (ii) It is trivial.

  (ii)$\Rightarrow $ (iv)If $E\not =\{0\}$  is not a reflexive Banach space, then $\mathcal{WAP}(F(E))=\{0\}$ (see  \cite{y1} Theorem 3). Thus $\mathcal{WAP}(B(E))$ can not separate  the points of $F(E)$ and so $B(E)$. Thus it's not a $\mathcal{WAP}$-algebra.

   (iv)$\Rightarrow $ (i)Let  $(T_{\alpha})$ be a net  in $B(E)$ an $T,S\in  B(E)$ such that $T_{\alpha}\stackrel{w^*}{\longrightarrow} T$. Then for $\tau=\sum_{n=1}^{\infty}\xi_n\otimes\xi_n'$ in $E\hat\otimes E^*$, predual of $B(E)$, we have:
     \begin{eqnarray*}
    <ST_{\alpha},\tau>&=&\sum_{n=1}^{\infty}<ST_{\alpha}(\xi_n),\xi_n'>\\
    &=&\sum_{n=1}^{\infty}<T_{\alpha}(\xi_n),S^*(\xi_n')>\longrightarrow\sum_{n=1}^{\infty}<T(\xi_n),S^*(\xi_n')>\\
    &=&\sum_{n=1}^{\infty}<ST(\xi_n),\xi_n'>=<ST,\tau>
 \end{eqnarray*}
 And similarly $<T_{\alpha}S,\tau>\longrightarrow<TS,\tau>$.
 Hence the multiplication on $B(E)$ is separately $w^*$-continuous and $B(E)$ is a dual Banach algebra.

  (iii)$\Rightarrow $ (iv)If $E\not =\{0\}$  is not a reflexive Banach space. Then  $\mathcal{WAP}(B(E))=\{0\}$(\cite[Theorem3]{y1}) and $E\hat\otimes E^*=\{0\}$.
   This is a contradiction.

    (iv)$\Rightarrow $ (iii) See  \cite[Theorem 1]{y1}.
    \end{proof}

The  Proposition \ref{equivalent}(iii)with Corollary \ref{WAP-algebra} shows that a $\mathcal{WAP}$-algebra ${\A}$ is  isomorphic to  a closed subalgebra  of a dual Banach algebra, so it is convenient to suppose that a  $\mathcal{WAP}$-algebra is norm closed subalgebra of some dual Banach algebra. This can be achieved by redefining suitable norm.

A large class of  $\mathcal{WAP}$-algebras are  semisimple cummitative Bananach algebras. We refer to \cite{Kaniuth} for this kind of algebras.
\begin{proposition}Let  $\A$ be a semisimple cummitative Bananach algebra, then $\A$ is $\mathcal{WAP}$-algebra.
\end{proposition}\begin{proof}
 Since $\Delta(\A)\subseteq WAP(\A)$  and $\Delta(\A)$ seprerates the points of $\A$ and $||a||_1=\sup\{|f(a)|: a\in \Delta(\A)\}$ is a Banach norm on $\A$. We know that every Banach norm on semisimple Banach algebras are equivalent (\cite{Kaniuth}). Thus $\A$ is a $WAP$- algebra. \end{proof} 
Let $G$ be a locally compact group and $\omega$ is a Borel-measurable weight function on it. Then the Fourier-Stieltjes algebra $B(G)$ and the  weighted measure algebra $M_b(G,\omega)$ are  dual Banach algebras, also the Fourier algebra $A(G)$ and $L^1(G,\omega)$ are   closed ideals  in them respectively. Hence all are $\mathcal{WAP}$-algebras.

  Let  $({\A},{\A}_*)$ be a dual Banach algebra and $I$ is an ideal in ${\A}$, we 
  define the annihilators: $$I^{\perp}=\{f\in {\A}^*:f(i)=0 ,\mbox{ for all}\quad i\in I\},\quad {}^{ \perp}I=\{f\in {\A}_*:f(i)=0 ,\mbox{ for all}\quad i\in I\}.$$
  It is easy to see that $({}^{ \perp}I)^{\perp}$ is the ${\omega^*}$-closure of $I$ in ${\A}$.

Let $||a||_{\mathcal{WAP}(\A)}=\sup\{|f(a)|:f\in \mathcal{WAP}(\A), ||f||\leq1\}$. Clearly $||.||_{\mathcal{WAP}(\A)}\leq ||.||$. The following lemma show that whenever the seminorm $||.||_{\mathcal{WAP}(\A)}$ could be a norm on $\A$.
\begin{lemma}\label{Waptt}Let $\A$ be a Banach algebra. Then the following conditions are equivalent.
\begin{enumerate}
\item  $\mathcal{WAP}(\A)$ separates the points of $\A$.
\item  $^\perp (\mathcal{WAP}(\A))=\{0\}$.
\item $\overline {\mathcal{WAP}(\A)}^{\sigma(\A^*,\A)}=\A^*$
\item $||.||_{\mathcal{WAP}(\A)}$ is a norm on $\A$.
\end{enumerate}
\end{lemma}
 \begin{proof} Clearly  (1) and (4)are equivalent.

(1)$\Rightarrow$(2).
Let $0\not=x\in \A$. Then there is a $f\in \mathcal{WAP}(\A)$ such that $f(x)\not=0.$ Thus $x\not\in {}^\perp (\mathcal{WAP}(\A))$, so ${}^\perp (\mathcal{WAP}(\A))=\{0\}$.

   (2)$\Rightarrow$(3). $\A^*=0^\perp=({}^\perp(\mathcal{WAP}(\A))^\perp=\overline {\mathcal{WAP}(\A)}^{\omega^*}$.

     (3)$\Rightarrow$(1). Let $\overline {\mathcal{WAP}(\A)}^{\omega^*}=\A^*$  then  for $0\not= x\in \A$ there is a  $f\in \A^*$ such that $f(x)\not=0$ . Let $f_\alpha\in \mathcal{WAP}(\A)$ convergent to $f$  in $\omega^*$ topology then there is an index $\alpha $ such that $f_\alpha(x)\not=0$ so $\mathcal{WAP}(\A)$ separates the points of $\A$.
 \end{proof}
  By using \cite[Theorem 1]{y1} and the previous lemma, the following is immediate.
\begin{proposition}The following statements are equivalent.
\begin{enumerate}
  \item $\A$ is $\mathcal{WAP}$-algebra.
  \item $\mathcal{WAP}(\A)$ separates the points of $\A$ and $||.||_{\mathcal{WAP}(A)}$ and $||.||$ are equivalent.
\end{enumerate}
\end{proposition}
\begin{theorem}
Let ${\A}=(\A_*)^*$ be a dual Banach  algebra. Let $I$ be a  closed ideal in ${\A}$, then the following statements are equivalent.
\begin{enumerate}
  \item[(i)]
 $\frac{{\A}}{I}$  is a dual Banach  algebra with respect to $^\perp I$.

  \item[(ii)] $I$ is $\omega^*$- closed ideal in $\A$.
\end{enumerate}
\end{theorem}
\begin{proof}
(ii)$\Rightarrow$(i) Let $I$ be $\omega^*$- closed ideal in $\A$. Then  $$\frac{A}{I}=\frac{\A}{\overline{I}^{\omega^*}}=\frac{\A}{(^\perp I)^\perp}=(^\perp I)^*$$
For $x\in \A$ and $f\in \A_*$ we  define the module action of $\frac{A}{I}$ over $^\perp I$ by $(x+I).f=x.f$ and $f.(x+I)=f.x$. It is routine to check that this module actions are well defined and $\frac{\A}{I}.(^\perp I)\subseteq (^\perp I)$ and  $(^\perp I).\frac{\A}{I}\subseteq (^\perp I)$, so $\frac{\A}{I}$ is a  dual Banach algebra with respect to $^\perp I$.

   (ii)$\Rightarrow$(i). $\frac{A}{I}=(^\perp I)^*=\frac{\A}{(^\perp I)^\perp}=\frac{\A}{\overline{I}^{\omega^*}}$ thus $(\frac{A}{I})^*=(\frac{A}{\overline {I}^{\omega^*}})^*$. So $^\perp(I^\perp)=^\perp(\overline{I}^{\omega^*})^\perp$  and thus  $I=\overline{I}^{\omega^*}$
   \end{proof}
  In the next theorem  ${\A}$ is a $\mathcal{WAP}$-algebra, so we suppose that  ${\A}$ is a closed subalgebra of $\mathcal{WAP}({\A})^*$.
\begin{theorem}
Let ${\A}$ be a $\mathcal{WAP}$- algebra.
\begin{enumerate}
  \item[(i)]
 Let $I$ be a  $\sigma(\A,\mathcal{WAP}(\A))$ closed  ideal in ${\A}$, then $\frac{{\A}}{I}$  is a $\mathcal{WAP}$- algebra.

  \item[(ii)] Let ${\B}$ be a Banach algebra and $\varphi:{\A}\rightarrow{\B}$ be an onto
      continuous  homomorphism, then ${\B}$ is a $\mathcal{WAP}$- algebra.
\end{enumerate}
\end{theorem}
\begin{proof}
(i)Let $\pi$  be the canonical map from
${\A}$ onto ${\A}/I $. It is well known that $\pi^*:({\A}/I)^*\longrightarrow{\A}^*$ is  an isometric isomorphism  onto it's rang  $I^{\perp}$.
 We conclude from latter isometry that,  $\mathcal{WAP}({\A}/I)\cong \mathcal{WAP}({\A})\cap I^{\perp}$ is the predual of $I$ in  the dual Banach algebra $\mathcal{WAP}({\A})^*$. 
 $ I$ is a $\omega^*$-closed ideal of $\mathcal{WAP}({\A})^*$ with respect to both Arens products.  So
 $$\mathcal{WAP}({\A})^*/{ I}\cong \mathcal{WAP}({\A}/I)^*$$ Since ${\A}$ is  a $\mathcal{WAP}$-algebra, then ${\A}/I$ is a  closed subalgebra of $\mathcal{WAP}({\A})^*/I$ and  so it is a  $\mathcal{WAP}$-algebra.

(ii) Let  $I=\ker(\varphi)$, then $I$ is  closed ideal in ${\A}$ and ${\A}/I\cong {\B}$
isometrically isomorphic.  Thus by (i) the Banach algebra ${\B}$ is a $\mathcal{WAP}$-algebra.
\end{proof}

Let ${\A}\oplus{\B}$ denote the direct sum of Banach algebras ${\A}$ and ${\B}$, and  endow it with coordinate wise  operation.
\begin{lemma}
\begin{enumerate}
  \item [(i)]If ${\A}$ and ${\B}$ are $\mathcal{WAP}$-algebras, then ${\A}\oplus{\B}$ is $\mathcal{WAP}$-algebras as well.
  \item [(ii)]If ${\A}$ and ${\B}$ are dual Banach algebras, then ${\A}\oplus{\B}$ is  dual Banach algebra as well.
  \item [(iii)]If ${\A}$ and ${\B}$ are Arens regular Banach algebras, then ${\A}\oplus{\B}$ is Arens regular as well.
\end{enumerate}
\end{lemma}
\begin{corollary}Let $\A^\sharp$ be the unitization of $\A$. Then:
\begin{enumerate}
  \item [(i)] ${\A}$ is a $\mathcal{WAP}$-algebras if and only if ${\A}^\sharp$ is $\mathcal{WAP}$-algebras as well.
  \item [(ii)] ${\A}$ is a dual Banach algebras if and only if ${\A}^\sharp$ is  dual Banach algebra as well.
  \item [(iii)]  ${\A}$ is  Arens regular Banach algebras if and only if ${\A}^\sharp$ is Arens regular as well.
\end{enumerate}
\end{corollary}
\begin{theorem}
Let ${\A}$ be a Banach algebra and  $I$ is  a  closed complemented ideal in ${\A}$.
 If $I$ and $\frac{{\A}}{I}$ are dual Banach algebra [respectively, $\mathcal{WAP}$-algebra, Arens regular Banach algebra ], then  ${\A}$ is a dual Banach  algebra [respectively, $\mathcal{WAP}$-algebra, Arens regular Banach algebra ].
\end{theorem}
\begin{proof}
(i): Let $I=E^*$ (respectively $\frac{{\A}}{I}=F^*$) where $E$ (respectively $F$) is an $I$-bimodule (respectively $\frac{{\A}}{I}$-bimodule).   Suppose that ${\A}_*:=E\oplus F$.   Then ${\A}_*^*\cong(E\oplus F)^*\cong E^*\oplus F^*\cong I\oplus \frac{{\A}}{I}\cong {\A}$. We can regard ${\A}_*$ as  an ${\A}$-bimodule by $$a.\mu:=(x.\mu_1,y.\mu_2)\quad and\quad \mu.a:=(\mu_1.x,\mu_2.y)$$ where $a=(x,y) \quad and\quad \mu=(\mu_1,\mu_2), x\in I,y\in \frac{{\A}}{I},\mu_1\in E,\mu_2\in F$.

(ii):Suppose that  $I$ is a closed subalgebra of $\mathcal{WAP}(I)^*$ and $\frac{{\A}}{I}$ is closed subalgebra of $\mathcal{WAP}({\A})^*/{ \bar I^{\omega^*}}\cong \mathcal{WAP}({\A}/I)^*$ then ${\A}\cong I\oplus \frac{{\A}}{I}$ is a closed subalgebra of $\mathcal{WAP}(I)^*\oplus \mathcal{WAP}({\A})^*/{\bar I^{\omega^*}}$. So ${\A}$  is a $\mathcal{WAP}$-algebra.

(iii): Suppose that  $\mathcal{WAP}(I)=I^*$   and  $\mathcal{WAP}({\A}/I)\cong ({\A}/I)^* $.  Then ${\A}^*\cong(I\oplus \frac{{\A}}{I})^*\cong I^*\oplus (\frac{{\A}}{I})^*\cong \mathcal{WAP}(I)\oplus \mathcal{WAP}(\frac{{\A}}{I})\cong \mathcal{WAP}({\A})$. The last equality is a result of  iterate limit argument. Let $F\in \mathcal{WAP}(I)\oplus \mathcal{WAP}(\frac{{\A}}{I})$ and $(x_n,y_n)$ ,$(a_m,b_m)$ are sequences in unit ball of $I\oplus \frac{{\A}}{I}$. Then $F=(f_1,f_2)$ such that $f_1\in \mathcal{WAP}(I)$ and $f_2\in \mathcal{WAP}(\frac{{\A}}{I})$. So that
\begin{eqnarray*}
  \lim_n\lim_m F(x_n,y_n).(a_m,b_m) &=& (\lim_n\lim_mf_1(x_na_m),\lim_n\lim_mf_2(y_nb_m)) \\
   &=&(\lim_m\lim_nf_1(x_na_m),\lim_m\lim_nf_2(y_nb_m)) \\
   &=& \lim_m\lim_n F(x_n,y_n).(a_m,b_m)
\end{eqnarray*} and $F\in \mathcal{WAP}({\A})$.
\end{proof}

Let $X$ and $Y$ be Banach spaces  and $T: X^*\rightarrow Y^*$ be a bounded operator. Recall that  the following conditions are equivalent:
\begin{enumerate}
    \item [(i)]$T=S^*$, for some bounded  operator $S:Y\rightarrow X$.
    \item [(ii)]$T^*(Y)\subseteq X$.
    \item [(iii)]$T$ is $\omega^*$-$\omega^*$-continuous operator.
\end{enumerate}

 We now state  the main result of this section.
\begin{theorem}\label{rty} Let $\A$ be a Banach algebra. Then the
following statements are equivalent:
\begin{enumerate}
 \item [(i)]$\A$ is  a Arens regular Banach algebra.
    \item [(ii)]$(\A^{**},\square)$ is a dual Banach algebra.
    \item [(iii)]$(\A^{**},\lozenge)$ is a dual Banach algebra.
    \item [(iv)]$(\A^{**},\square)$ is a $\mathcal{WAP}$-algebra.
     \item [(v)]$(\A^{**},\lozenge)$ is a $\mathcal{WAP}$-algebra.
     \item[(vi)]$\A$ is  a $\mathcal{WAP}$- Banach algebra and $\mathcal{WAP}(\A)$ is a $\omega^*$-closed set in $\A^*$.
\end{enumerate}
\end{theorem}
\begin{proof}
  If $\A$ is Arens regular, then by \cite[Proposition 1.2]{ARV} $\A^{**}$ is dual Banach algebra  so(ii) and  (iii) are equivalent. By \cite[Theorem 4.10]{kas} there is a reflexive Banach space $E$  such that $\A$ and also $\A^{**}$ are closed subalgebras of $B(E)$. Hence $\A^{**}$ is  a $\mathcal{WAP}$-algebra. Thus (i)$\Rightarrow$ (ii)$\Leftrightarrow$ (iii)$\Rightarrow$  (iv) were established.

  We show that if $\A^{**}$ with either Arens product has a  representation  on a reflexive Banach space  then it is a  representation with another Arens product. Thus the equality (iv)$\Leftrightarrow$ (v) was  established.
 Let $\pi: (\A^{**},\square)\rightarrow B(E)$  be an isometric representation on a reflexive Banach space $E$ (See  corollary\ref{equiva} and explanations after it).  By \cite[Proposition 4.2]{Da2} $\pi$ is $\omega^*$-$\omega^*$-continuous.
  Let $\Phi,\Psi\in \A^{**}$ and $(a_{\alpha})$ and $(b_{\beta})$ be nets in $\A$  that  $\omega^*$-convergening to  $\Phi$ and $\Psi$, respectively. Then for each $\tau\in E\hat\otimes E^*$
  \begin{eqnarray*}
    <\tau,\pi(\Phi\square\Psi)> &=& \lim_{\alpha}\lim_{\beta}<\tau,\pi(a_{\alpha})\circ\pi(b_{\beta})> \\
    &=& \lim_{\beta}\lim_{\alpha}<\tau,\pi(a_{\alpha})\circ\pi(b_{\beta})> \quad\quad (\mbox {since  }  \tau\in \mathcal{WAP}(B(E)))\\
     &=& <\tau,\pi(\Phi\lozenge\Psi)>
  \end{eqnarray*}
  So $\pi: (\A^{**},\lozenge)\rightarrow B(E)$ is an isometric  representation. Similarly for another  Arens product.

(v)$\Rightarrow$ (i)Let $(\A^{**},\lozenge)$ is  a $\mathcal{WAP}$-algebra. Then
there is an isometric representation $\pi: (\A^{**},\lozenge)\rightarrow B(E)$, for some reflexive Banach space $E$.
Hence by above equalities,
$$\pi(\Psi\lozenge\Phi)=\pi(\Psi)\circ\pi(\Phi)=\pi(\Psi\square\Phi)\quad(\Phi,\Psi\in
\A^{**})$$
 Since  $\pi$ is injective, we conclude
  $\A$ is Arerns regular Banach
algebra.

(vi)$\Rightarrow$ (i) Let $\A$ be a $\mathcal{WAP}$-algebra and $\mathcal{WAP}(\A)$ be a $\omega^*$-closed set in $\A^*$.  By lemma \ref{Waptt} $\overline {\mathcal{WAP}(\A)}^{\sigma(\A^*,\A)}=\A^*={\mathcal{WAP}(\A)}$. 

(i)$\Rightarrow$ (vi) Let  $\A$ be a a regular Banach algebra. Since $\A$ is a closed subset of $\A^{**}$,  $\A$ is a $\mathcal{WAP}$-algebra.  and $\overline {\mathcal{WAP}(\A)}^{\sigma(\A^*,\A)}=\A^*={\mathcal{WAP}(\A)}$. So
${\mathcal{WAP}(\A)}$ is a $\omega^*$-closed subset of $\A^*$.
\end{proof}
\begin{proposition}\label{ttt}
 $\A^{**}$ is a Arens regular Banach algebra if and only if $\A$ is a Arens regular Banach algebra and ${\mathcal{WAP}(\A^{**})}={\mathcal{WAP}(\A)}^{**}$.
\end{proposition}
\begin{proof}
Let  $\A^{**}$ be a Arens regular Banach algebra, then  $\mathcal{WAP}(\A^{**})=\A^{***}$. Also $\A$ is a closed subalgebra of $\A^{**}$ so it is  Arens regular Banach algebra. Thus  $\mathcal{WAP}(\A)=\A^*$ and so  $\mathcal{WAP}(\A^{**})={\mathcal{WAP}(\A)}^{**}$.

Conversely,  let $\A$ be a Arens regular Banach algebra then  ${\mathcal{WAP}(\A)}=\A^*$. Thus  ${\mathcal{WAP}(\A^{**})}={\mathcal{WAP}(\A)}^{**}=\A^{***}$, and so $\A^{**}$ is Arens regular.
\end{proof}
There exists a Banach algebra $\A$ such that $\A$ is Arens regular, but $\A^{**}$ is not Arens regular (see \cite{DL}).
\begin{examples}
\begin{enumerate}
\item Let $S=(\mathbb{N},\min)$. Then  $\ell_1(S)$ is a $\mathcal{WAP}$-algebra, but the second dual isn't so.
\item Let $S=(\mathbb{N},\max)$. Then  $\ell_1(S)$ is a dual Banach algebra, but the second dual isn't so.
\item Let $S$ be infinite zero semigroup, then   $\ell_1(S)^{**}$ is a dual Banch algebra but   $\ell_1(S)$ isn't so.
\end{enumerate}
\end{examples}

\section{ absolutely semigroup  measure algebras}

Following \cite{BJM}, a semitopological semigroup is a semigroup $S$  equipped with a Hausdorff  topology under which the multiplication of $S$ is separately continuous. If the multiplication of $S$ is jointly continuous then $S$ is said to be a topological semigroup. We write $\ell^{\infty}(S)$  for the commutative algebras of all bounded complex-valued   functions on $S$. In the case where $S$ is  locally compact  we also write $C(S)$ and $C_0(S)$ for the closed subalgebras of $\ell_\infty(S)$ consist of continuous elements and continuous elements  which vanish  at infinity, respectively. We also denote the space of all {\it weakly almost periodic} functions on $S$ by $\mathcal{WAP}(S)$ which is defined by $$\mathcal{WAP}(S)=\{f\in C(S): \{R_sf: s\in S\}\ \mbox{ is relatively weakly compact in}\quad C(S)\},$$ where $R_sf(t)=f(ts), \ (s,t\in S).$ Then $\mathcal{WAP}(S)$ is a closed subalgebra of $C(S)$. Many other  properties of $\mathcal{WAP}(S)$ and its inclusion relations among other function algebras are completely explored in \cite{BJM}.\\
Let $M_b(S)$ denotes the Banach space of all
bounded complex regular Borel measures on $S$ with the total variation norm.
Then $M_b(S)$ with the convolution multiplication $*$ defined by the equation
$$\langle  \mu*\nu,f\rangle=\int_S\int_Sf(xy)d\mu(x)d\nu(y)\quad (f\in C_0(S))$$
is a convolution measure algebra as the dual of $C_0(S)$.

 Let $\mu\in M_b(S)$ and let $L^\infty (|\mu|)$ be the Banach space of all bounded Borel-measurable functions on $S$ with essential supremum norm,$$||f||_{\mu}=\inf\{\alpha\geq 0:\{x\in S:|f(x)|>\alpha\}\quad\mbox{is}\quad |\mu|-\mbox{null}\}.$$ Consider the product linear space $\Pi\{L^\infty(|\mu|):\mu\in M_b(S)\}$. An element $f=(f_\mu)$ in this product is called a generalized function on $S$ if the following conditions are satisfied:
\begin{enumerate}
  \item $||f||_\infty=\sup\{||f||_\mu:\mu\in M_b(S)\}$ is finite.
  \item if $\mu,\nu\in M_b(S)$ and $|\mu|\ll|\nu|$, then $f_\mu=f_\nu$, $|\mu|$-a.e.
\end{enumerate}
Let $GL(S)$ denote the linear subspace of all generalized functions on $S$. Then $M_b(S)^*$ is isometrically isomorphic to $GL(S)$, see \cite{Lashkarizadeh}.
We identify  $C_0(S)$ with  its image by canonical map $\iota: C_0(S)\rightarrow GL(S)$. The $w^*$-topology on $GL(S)$ is $\sigma(GL(S),M_b(S))$.
Let $\A$ be a convolution  measure algebra in $M_b(S)$, i.e. a norm-closed solid subalgebra of $M_b(S)$. In other words, for  all $\mu\in M_b(S)$ and $\nu\in \A$ such that $|\mu|\ll|\nu|$, we have $\mu\in \A$.
The space of all measures $\mu\in  M_b(S)$ for which the maps $x \mapsto\delta_x*|\mu|$ and
$x \mapsto|\mu|*\delta_x$ from $S$ into $M_b(S)$ are weakly continuous is denoted by $M_a(S)$
, where $\delta_x$ denotes the Dirac measure at $x$. Then $M_a(S)$ is
a closed two-sided L-ideal of $M_b(S)$ ; see Baker and Baker \cite{BAKER} or Dzinotyiweyi
\cite{DZINO}. The locally compact semigroup $S$ is called foundation if , $F_a(S)$, the closure of the
set $$\cup\{supp(\mu): \mu\in M_a(S)\}$$  coincide with $S$ .

Let $F$ and $K$ be nonempty subsets of a semigroup $S$ and $s\in S$. We put $$s^{-1}F=\{t\in S:st\in F\}\mbox{, and}\quad Fs^{-1}=\{t\in S:ts\in F\}$$
and we also write $s^{-1}t $ for the set $s^{-1}\{t\}$, $FK^{-1}$
for $\cup\{Fs^{-1}:s\in K\}$ and $K^{-1}F$ for
$\cup\{s^{-1}F:s\in K\}$.

 The authors and H.R. Ebrahimi- Vishki in \cite{Bkh}, for a locally compact topological semigroup $S$  showed that $M_b(S)$ is a $\mathcal{WAP}$-algebra if and only if  the canonical  mapping $R:S\longrightarrow S^{wap}$ is one to one, where $S^{wap}$ is the weakly almost periodic compactification of $S$. In this section we show that for a foundation semigroup   $S$, this is equivalent to $M_a(S)$ is a $\mathcal{WAP}$-algebra.

 The following theorem is the main result of this section.
\begin{theorem}\label{wa}
Let $S$ be a foundation locally compact topological semigroup with identity. Then $M_a(S)$ is a $\mathcal{WAP}$-algebra if and only if $M_b(S)$ is so.
\end{theorem}
\begin{proof}
If $M_b(S)$ is a $\mathcal{WAP}$-algebra then $M_a(S)$ is a closed subalgebra in it, so $M_a(S)$ is a $\mathcal{WAP}$-algebra.

 As  the proof of lemma 1.8 \cite[p.16]{DZINO} for all  $x_1,x_2\in S$ with $x_1\not=x_2$ there is $\mu\in M_a(S)$ such that  \[\delta_{x_1}*\mu\not=\delta_{x_2}*\mu.\]
Let $M_a(S)$  is a $\mathcal{WAP}$-algebra, then for all $\mu\in M_a(S)$ we have $||\mu||=\sup\{|\langle f,\mu\rangle|: ||f||\leq1, f\in \mathcal{WAP}(M_a(S))\}$ and by \cite[Theorem3.7]{Lashkarizadeh} $\mathcal{WAP}(S)= \mathcal{WAP}(M_a(S))$, so  $||\mu||=\sup\{|\langle f,\mu\rangle|: ||f||\leq1, f\in \mathcal{WAP}(S)\}$.  Thus there  is a  $f\in \mathcal{WAP}(S)$ such that
 \[\langle\mu,R_{x_1}f\rangle=\langle \delta_{x_1}*\mu,f\rangle\not=\langle\delta_{x_2}*\mu,f\rangle=\langle\mu,R_{x_2}f\rangle\]
 Thus  $R_{x_1}\not=R_{x_2}$. This means the canonical  mapping $R:S\longrightarrow S^{wap}$ is one to one, where $S^{wap}$ is the weakly almost periodic compactification of $S$.
  By \cite[Theorem 3.1]{Bkh} $M_b(S)$ is a $\mathcal{WAP}$-algebra.
 \end{proof}
 By using \cite[Theorem3.1]{Bkh} and the previous Theorem, the following is immediate.
\begin{corollary}
Let $S$ be a foundation semigroup. The following statements are equivalent.
\begin{enumerate}
\item $M_a(S)$ is $\mathcal{WAP}$-algebra.
\item $\ell_1(S)$ is a $\mathcal{WAP}$-algebra.
\item $M_b(S)$ is a $\mathcal{WAP}$-algebra.
\end{enumerate}
\end{corollary}
By using \cite[Corollary8]{RJ} and Theorem\ref{rty}, the following is immediate.
\begin{corollary}
Let $S$ be a foundation semigroup. The following statements are equivalent.
\begin{enumerate}
\item $M_a(S)^{**}$ is $\mathcal{WAP}$-algebra.
\item $\ell_1(S)^{**}$ is a $\mathcal{WAP}$-algebra.
\item $M_b(S)^{**}$ is a $\mathcal{WAP}$-algebra.
\end{enumerate}
\end{corollary}

 The next example shows that in  Theorem\ref{wa} the condition $S$ be foundation semigroup can not be omitted.
\begin{examples}\begin{enumerate}\item[(a)] Let $S$ be a locally compact semigroup.  If $C_0(S)\subseteq \mathcal{WAP}(S)$ then $M_a(S)$ is a $\mathcal{WAP}$-algebra.
\item[(b)] Let $S$ be a compact semigroup, then $M_a(S)$ is a $\mathcal{WAP}$-algebra.
 \item[(c)]  Let $S=\Bbb R^2$ and   $(x,y).(\alpha,\beta)=(\alpha x,\alpha y+\beta)$ with the standard Euclidean topology. Then $ M_a(S)=\{0\}$ and so $S$ is not a foundation semigroup. In fact,  let $x_n=(0,1+\frac{1}{n})$, $x_0=(0,1)$, $K=[-1,1]\times [-1,1]$ and $\mu\in M_a^r(S)$. Then  $Kx_n^{-1}=\emptyset$ and  $Kx_0^{-1}=S$. Then by continuity of $\mu$, we have $0=|\mu|(Kx_n^{-1})\rightarrow |\mu|(Kx_0^{-1})=||\mu||$. So $ M_a(S)=\{0\}$.

 Let  $f\in C_0(S)$ and $f(0,c)\not=0$  for some $c\in \mathbb{R}$. Suppose  $a_n=(0,n)$ and $b_m=(1,\gamma_m)$ where $\gamma_m$ is on the line $y=-n x+c$. Then
$$\lim_n\lim_mf(a_nb_m)=f(0,c)\not=\lim_m\lim_nf(a_nb_m)=0$$
 So by examining the iterate limits we certify that $f\not\in \mathcal{WAP}(S)$. Thus $\mathcal{WAP}(S)$ does not separate the points of $Y$ and so $M_b(S)$. This show that the map $R:S\rightarrow S^{wap}$ is not one to one. Therefore  $M_b(S)$ isn't a $\mathcal{WAP}$- algebra ( see \cite{Bkh} also) but clearly $M_a(S)$ is a  $\mathcal{WAP}$- algebra.
\end{enumerate}
\end{examples}

\noindent{{\bf Acknowledgments.}}  This research was supported by
the Center of Excellence for Mathematics at the Isfahan university.


\begin{thebibliography}{00}
\bibitem{ARV} M. Abolghasemi, A. Rejali, \and H.R. Ebrahimi Vishki, {\em Weighted semigroup algebras as dual Banach algebras,} arXiv:0808.1404 [math.FA].
\bibitem{AKhR} F. Abtahi, B. Khodsiani,  \and A. Rejali,  {\em Arens regularity of inverse semigroup algebras }, Bull. Iranian Math. Soc. Vol. 40 (2014), No. 6, pp. 1527-1538.

\bibitem{BAKER} A.C. Baker, J.W. Baker, {\em Algebras of measures on a locally compact semigroups,} \textbf{1}, J.London Math., no.1 (1969), 249-259.


\bibitem{BJM} J.F. Berglund, H.D. Junghenn \and P. Milnes, {\em Analysis on Semigroups}, Wiley, New York, (1989).

\bibitem{B} R.B. Burckel, {\em Weakly Almost Periodic Functions on Semigroups,} Gordon and Breach, New york,(1970).


\bibitem{D}H.G. Dales, {\em  Banach Algebras and Automatic Continuity,} Clarendon Press, Oxford, 2000.

\bibitem{DL} H.G. Dales, A.T.M. Lau, {\em The Second Dual of Beurling Algebras,} Mem. Amer. Math.
Soc. \textbf{177} (2005), no. 836.



\bibitem{Da2}M. Daws, {\em Dual Banach algebras: Representation and injectivity,} Studia Math, \textbf{178}(2007), 231-275.

\bibitem{DPW} M. Daws, H.L. Pham \and S. White, {\em Conditions implying the uniqeness of the weak$^*-$topology on certain group algebras}, Houston J. Math. \textbf{35} (2009), no. 1, 253-276.

\bibitem{Dzino} H.A.M. Dzinotyiweyi, {\em Weakly almost periodic functions and the irregularity of multiplication in semiogroup algebras,} Math. Scand.\textbf{46}, (1980), 157-172.

\bibitem{DZINO} H.A.M. Dzinotyiweyi, {\em The Analogue of the Group Algebra for Topological Semigroups,} Reserch notes in mathematics, Pitman Publishing,(1984).



\bibitem{kas} S. Kaijser, {\em On Banach modules I}, Math. Proc. Cambridge Philos. Soc. \textbf{90} (1981).

\bibitem{Kaniuth} E.Kaniuth,  {\em A Course in Commutative  Banach Algebras} Springer New York 2008.

\bibitem{Bkh}H.R. Ebrahimi Vishki, B. Khodsiani  \and A. Rejali, {\em Weighted semigroup measure algebra as a  $\mathcal{WAP}$-algebra,} U.P.B. Scientific Bulletin Accepted for publication.

\bibitem{Lashkarizadeh} M. Lashkarizadeh Bami, {\em Function algebras on weighted topological semigroups,} Math. Japonica \textbf{47}, no.2 (1998), 217-227.



\bibitem{RJ}A. Rejali, {\em The Arens regularity of weighted semigroup algebras,} Sci. Math. Japon,  \textbf{60}, no.1 (2004), 129-137.

\bibitem{Ulger} A. Ulger, {\em Continuity of weakly almost periodic functionals on $L^1(G)$}, Quart. J. Math. Oxford(2), 37 (1986), 495-497.



\bibitem{y1}N.J.Young, Periodicity of functionals and representations of normed algebras on reflexive spaces., Proc. Edinburgh Math. Soc.(2){\bf20} (1976/77), 99-120.


\end{thebibliography}
\end{document}